\documentclass[12 pt]{amsart}
\usepackage{graphicx}
\usepackage{amsmath}
\usepackage{graphicx}
\usepackage{mathrsfs}

\usepackage[active]{srcltx}
\newtheorem{theorem}{Theorem}[section]

\newtheorem{lemma}[theorem]{Lemma}

\newtheorem{remark}[theorem]{Remark}

\numberwithin{equation}{section}

\let\mathcal=\mathscr

\def\L1#1{L^1(#1)}

\def\L#1#2{L^{#1}(#2)}

\def\lef({\left(}
\def\rig){\right)}

\def\nbx#1{\nabla\phi(x(t))}

\usepackage{amssymb}
\usepackage{relsize}

= 600pt \footskip=40pt\evensidemargin=10pt
\begin{document}
\title[Asymptotic for a second order evolution equation]
{Asymptotic for a second order evolution equation with vanishing damping term
and Tikhonov regularization}%
\author{Mounir Elloumi, Ramzi May and Chokri Mnasri}%
\address{Mathematics Department, College of Science, King Faisal University, P.O. 380, Ahsaa 31982, Kingdom of Saudi Arabia}
\address{Mathematics Department, College of Science, Sfax University, Sfax 3042, Tunisia}
\email{melloumi@kfu.edu.sa}
\address{Mathematics Department, College of Science, King Faisal University, P.O. 380, Ahsaa 31982, Kingdom of Saudi Arabia}
\address{Mathematics Department, College of Science of Bizerte, Carthage University, Bizerte, Tunisia}
\email{rmay@kfu.edu.sa}
\address{Mathematics Department, College of Science, King Faisal University, P.O. 380, Ahsaa 31982, Kingdom of Saudi Arabia}
\email{cmnasri@kfu.edu.sa}
\thanks{This work is supported by the Deanship of Scientific Research at King Faisal University under the Project 170065​}
 \vskip 0.2cm

\vskip 0.2cm
\date{}
\dedicatory{}
\begin{abstract}
Let $\mathcal{H}$\ be a real Hilbert space. We
investigate the long time behavior of the trajectories $x(.)$ of the vanishing
damped nonlinear dynamical system with Tikhonov regularizing term%
\begin{equation}
x^{\prime\prime}(t)+\gamma(t)x^{\prime}(t)+\nabla\Phi(x(t))+\varepsilon
(t)\nabla U(x(t))=0, \tag{GAVD$_\gamma,\varepsilon$}%
\end{equation}
where $\Phi,U:\mathcal{H}\rightarrow\mathbb{R}$ are two convex continuously
differentiable functions, $\varepsilon(.)$ is a decreasing function satisfying
$\displaystyle\lim_{t\rightarrow+\infty}\varepsilon(t)=0,$ and $\gamma(.)$ is a nonnegative
function which behaves, for $t$ large enough, like $\displaystyle\frac{K}{t^{\theta}}$
where $K>0$ and $0\leq\theta\leq1.$ The main contribution of this paper is the
following control result: If $\displaystyle\int_{0}^{+\infty}\frac{\varepsilon(t)}%
{\gamma(t)}dt=+\infty,$ $U$ \ is strongly convex and its unique minimizer
$x^{\ast}$ is also a minimizer of $\Phi$ then every trajectory $x(.)$ of
(GAVD$_{\gamma,\varepsilon}$) converges strongly to $x^{\ast}$ and the rate of
convergence of its energy function $$W(t)=\frac{1}{2}\left\Vert x^{\prime
}(t)\right\Vert ^{2}+\Phi(x(t))-\min\Phi$$
is of order to $\circ(1/t^{1+\theta})$. Moreover, we prove a new result concerning the weak convergence of the trajectories
of (GAVD$_{\gamma,\varepsilon}$) to a common minimizer of $\Phi$ and $U$ (if
one exists) under a simple condition on the speed of decay of the Tikhonov
factor $\varepsilon(t)$ to $0$ with respect to $\gamma(t).$
\end{abstract}
\maketitle

\section{Introduction and statement of the main results}
Let $\mathcal{H}$ be a real Hilbert space endowed with the inner product
$\langle.,.\rangle$ and the associated norm $\left\Vert .\right\Vert .$ Let
$\Phi,U:\mathcal{H}\rightarrow\mathbb{R}$ be two convex continuously differentiable
functions and $\gamma,\varepsilon$ be two real positive functions defined on a
fixed time interval $[t_{0},+\infty)$ for some $t_{0}>0.$ Motivated by the recent work \cite{ACR} of Attouch, Chbani, and Riahi on the asymptotic
behavior of the trajectories of the asymptotic vanishing damping dynamical
system with regularizing Tikhonov term%

\begin{equation}
x^{\prime\prime}(t)+\frac{\alpha}{t}x^{\prime}(t)+\nabla\Phi(x(t))+\varepsilon
(t)x(t)=0, \tag{AVD$_\alpha,\varepsilon$}%
\end{equation}
we investigate in this paper the long time behavior, as $t\rightarrow+\infty,$
of the trajectories of the following generalized version of the
(AVD$_{\alpha,\varepsilon}$) dynamical system%
\begin{equation}
x^{\prime\prime}(t)+\gamma(t)x^{\prime}(t)+\nabla\Phi(x(t))+\varepsilon
(t)\nabla U(x(t))=0. \tag{GAVD$_\gamma,\varepsilon$}%
\end{equation}
For the importance and the applications of these two dynamical systems and many
other related dynamical systems in Mechanics and Optimization, we refer the
reader to \cite{ACa}, \cite{AC}, \cite{AGR}, \cite{SBC} and references therein.

Throughout this paper, we assume the following general hypothesis:

\begin{description}
\item[H$_{1}$] The functions $\Phi,U:\mathcal{H}\rightarrow\mathbb{R}$ are
convex, differentiable, and bounded from below. We set $\displaystyle\Phi^{\ast}=\inf
_{x\in\mathcal{H}}\Phi(x)$ and $\displaystyle U^{\ast}=\inf_{x\in\mathcal{H}}U(x).$

\item[H$_{2}$] The set $S_{\Phi}:=\mbox{argmin}\Phi=\{z\in\mathcal{H}:~\Phi(z)=\Phi^{\ast}\}$
is not empty.

\item[H$_{3}$] The gradient functions $\nabla\Phi$ and $\nabla U$ of $\Phi$
and $U$ are Lipschitz on bounded subsets of $\mathcal{H}.$

\item[H$_{4}$] The function $\displaystyle\gamma:[t_{0},+\infty)\rightarrow
(0,+\infty)$ is absolutely continuous and satisfies the following
property: there exist $t_{1}\geq t_{0}$ and two real constants $K_{1},K_{2}>0
$ such that $$\displaystyle\gamma(t)\geq\frac{K_{1}}{t} \mbox{ and } \displaystyle\gamma^{\prime}(t)\leq
\frac{K_{2}}{t^{2}}$$ for almost every $\displaystyle t\geq t_{1}.$

\item[H$_{5}$] The function $\displaystyle\varepsilon:[t_{0},+\infty)\rightarrow(0,+\infty)$
is absolutely continuous, nonincreasing function and satisfies
$$\displaystyle\lim_{t\rightarrow+\infty}\varepsilon(t)=0.$$
\end{description}

Proceeding as in the proof of [Theorem 3.1, \cite{AGR}] and using the
classical Cauchy-Lipschitz and the energy function
\begin{equation}
W(t)=\frac{1}{2}\left\Vert x^{\prime}(t)\right\Vert ^{2}+\Phi(x(t))-\Phi
^{\ast}+\varepsilon(t)(U(x(t))-U^{\ast}), \label{energy}%
\end{equation}
one can easily prove that for every initial data $(x_{0},v_{0})\in
\mathcal{H}\times\mathcal{H},$ the dynamical system (GAVD$_{\gamma
,\varepsilon}$) has a unique solution $x(.)\in C^{2}(t_{0},+\infty;\mathcal{H})$
which satisfies $x(t_{0})=x_{0}$ and $x^{\prime}(t_{0})=v_{0}.$ Therefore, we assume in what follows that $x(.)$ is a classical global solution of
(GAVD$_{\gamma,\varepsilon}$) and we focus our attention on the study of the
long time behavior of $x(t)$ as $t$ goes to infinity. Before setting the
contributions of this work in this direction, let us first recall some well
known results on the asymptotic behavior of solutions of a variant dynamical
systems related to (GAVD$_{\gamma,\varepsilon}$). In the pioneer work
\cite{Al}, Avarez considered the case where $\gamma(.)$ is constant and $\varepsilon
=0.$ He established that the trajectory $x(t)$ converges weakly to some
element $\bar{x}$ of $S_\Phi$. In this case, the rate of convergence of
$\Phi(x(t))$ to $\Phi^{\ast}$ is of order $\circ(1/t)$ (see \cite{ACa})$.$To
overcome the drawback of the weak convergence to a non identified minimizer of
$\Phi$, Attouch and Cazerniki \cite{AC} proved that, up to adding a Tikhonov
regularizing term $\varepsilon(t)x(t)$ with $\displaystyle\int_{t_{0}}^{+\infty}%
\varepsilon(t)dt=+\infty,$ any trajectory $x(t)$ of the system
\begin{equation}
x^{\prime\prime}(t)+\gamma x^{\prime}(t)+\nabla\Phi(x(t))+\varepsilon(t)x(t)=0
\label{AC}%
\end{equation}
converges strongly to the element $x^{\ast}$ of minimum norm of the set
$S_\Phi.$ Using a different approach, Jendoubi and May \cite{JM10}
proved that this strong convergence result remains true even a perturbation
integrable term $g(t)$ is added to the equation (\ref{AC})$.$ In other
direction, in order to improve the rate of convergence of $\Phi(x(t))$ to
$\Phi^{\ast},$ Su, Boyd, and Candes \cite{SBC}\ introduced the dynamical
system which is the continuous version of the Nestrov's accelerated
minimization method \cite{Ne}%
\begin{equation}
x^{\prime\prime}(t)+\frac{\alpha}{t}x^{\prime}(t)+\nabla\Phi(x(t))=0.
\label{Nest}%
\end{equation}
They proved that if $\alpha\geq3$ then $$\Phi(x(t))-\Phi^{\ast}=O(1/t^{2}).$$
This result was later improved in \cite{ACPR} and \cite{Ma17}. In fact it was
proved that if $\alpha>3$ then $x(t)$ of converges weakly to some element
$\bar{x}$ of $S_\Phi$ and that $$\Phi(x(t))-\Phi^{\ast}=\circ(1/t^{2}).$$

Recently, in order to benefit at the same time of the quick minimization
property $\displaystyle\Phi(x(t))-\Phi^{\ast}=\circ(1/t^{2})$ due to the presence of the
vanishing damping term $\displaystyle\gamma(t)=\frac{\alpha}{t}$ in (\ref{Nest}) and the
strong convergence of the trajectories of (\ref{AC}) to a particular minimizer
of $\Phi$ which is a consequence of the regularizing Tikhonov term
$\varepsilon(t)x(t),$ Attouch, Chbani, and Riahi \cite{ACR} have considered
the dynamical system (AVD$_{\alpha,\varepsilon}$) and they have established
some properties of the asymptotic behavior of its trajectories which we can
summarize in the following theorem.

\begin{theorem}
[Attouch, Chbani and Riahi]\label{ThACR}Let $x\in C^{2}(t_{0},+\infty
;\mathcal{H})$ be a solution of \ (AVD$_{\alpha,\varepsilon}$).

\begin{description}
\item[A] If $\alpha>1$ and $\displaystyle\int_{t_{0}}^{+\infty}\frac{\varepsilon(t)}%
{t}dt<+\infty,$ then $\displaystyle\int_{t_{0}}^{+\infty}\frac{\left\Vert x^{\prime
}(t)\right\Vert ^{2}\text{\thinspace\thinspace}}{t}dt<+\infty,$ $\displaystyle\lim
_{t\rightarrow+\infty}x^{\prime}(t)=0$ and $\displaystyle\lim_{t\rightarrow+\infty}%
\Phi(x(t))=\Phi^{\ast}.$

\item[B] If $\alpha>3$ and $\displaystyle\int_{t_{0}}^{+\infty}t\varepsilon(t)dt<+\infty,$
then $x(t)$ converges weakly to some element of $S_\Phi$.\\Furthermore, the
associated energy function $\displaystyle W(t)=\frac{1}{2}\left\Vert x^{\prime
}(t)\right\Vert ^{2}+\Phi(x(t))-\Phi^{\ast}$ satisfies $\displaystyle W(t)=\circ(1/t^{2})$
and $\displaystyle\int_{t_{0}}^{+\infty}tW(t)dt<+\infty.$

\item[C] If the function $\varepsilon$ satisfies moreover one of the following hypothesis

\begin{description}
\item[H$_{5a}$] $\displaystyle\lim_{t\rightarrow+\infty}t^{2}\varepsilon(t)=+\infty$ if
$\alpha=3$

\item[H$_{5b}$] $\displaystyle t^{2}\varepsilon(t)\geq c>\frac{4}{9}\alpha(\alpha-3)$ if
$\alpha>3$

\item[H$_{5c}$] $\displaystyle\int_{t_{0}}^{+\infty}\frac{\varepsilon(t)}{t}dt=+\infty$
\end{description}
\end{description}
then $\displaystyle\liminf_{t\rightarrow+\infty}\left\Vert x(t)-x^{\ast}\right\Vert =0$
where $x^{\ast}$ is the element of minimal norm of the set $S_\Phi.$
\end{theorem}
In this paper, we improve and extend these results to the general dynamical system
(GAVD$_{\gamma,\varepsilon}$). Moreover, we discover some new asymptotic
properties of the trajectories of (GAVD$_{\gamma,\varepsilon}$).

Our first result is a general minimization property of (GAVD$_{\gamma
,\varepsilon}$) which is a slight improvement of the assertion (A) in the
previous theorem.

\begin{theorem}
[A general minimization property of (GAVD$_{\gamma,\varepsilon}$%
)]\label{General}Let $x(.)$ be a classical solution of (GAVD$_{\gamma
,\varepsilon}$). Then $\displaystyle\int_{t_{0}}^{+\infty}\gamma(t)\left\Vert x^{\prime
}(t)\right\Vert ^{2}dt<+\infty,$ and the energy function $W(t),$ defined by
(\ref{energy}), converges to $0$ as $t\rightarrow+\infty$. In particular
$\displaystyle\lim_{t\rightarrow+\infty}x^{\prime}(t)=0$ and $\displaystyle\lim_{t\rightarrow+\infty
}\Phi(x(t))=\Phi^{\ast}.$
\end{theorem}

Our second main result concerns the weak convergence property of the
trajectories of (GAVD$_{\gamma,\varepsilon}$). The first part of this result
is similar to the assertion B in Theorem \ref{ThACR}. \ Our proof, which
is different from the arguments given by Attouch, Chbani, and Riahi [Theorem
3.1, \cite{ACR}], provide an other confirmation of the fact, noticed
recently in many works as \cite{ACR}, \cite{ACa}, \cite{Ma17} and \cite{SBC},
that the value $\alpha=3$ in the the system (\ref{Nest}) is critical and
somehow mysterious. The second part of the theorem is a simple result on the
weak convergence to a common minimizer of the two convex functions $\Phi$ and
$U$ which, at our knowledge, is not known even in the case where the damping
term $\gamma$ is constant. A comparable result was proved by Cabot(see [Proposition 2.5, \cite{Ca}]) for the
first order system $x^{\prime}(t)+\nabla\Phi(x(t))+\varepsilon(t)\nabla
U(x(t))=0$.

\begin{theorem}
[Weak convergence properties of (GAVD$_{\gamma,\varepsilon}$)]\label{Weak} Assume that there exist $t_{1}\geq t_{0},~0\leq\theta\leq1,\alpha>0$ with
$\alpha>3$ if $\theta=1$ such that%
\begin{equation}
\gamma(t)\geq\frac{\alpha}{t^{\theta}}\text{ for every }t\geq t_{1}\text{ and
}\int_{t_{0}}^{+\infty}\left[  \left(  t^{\theta}\gamma(t)\right)  ^{\prime
}\right]  ^{+}dt<+\infty\label{assump}%
\end{equation}
where $\left[  \left(  t^{\theta}\gamma(t)\right)  ^{\prime}\right]  ^{+}%
=\max\{0,\left(  t^{\theta}\gamma(t)\right)  ^{\prime}\}.$ Let $x(.)$ be a
classical solution of (GAVD$_{\gamma,\varepsilon}$). Then the two following
properties hold:

\begin{description}
\item[P$_{1}$] If $\displaystyle\int_{t_{0}}^{+\infty}t^{\theta}\varepsilon(t)dt<+\infty$
then $x(t)$ converges weakly to some element of $S_\Phi.$

\item[P$_{2}$] If $S_\Phi\cap S_U\neq\varnothing$ and $\displaystyle\liminf_{t\rightarrow+\infty}t^{1+\theta}\varepsilon(t)>0~$then $x(t)$ converges
weakly to some element of $S_\Phi\cap S_U.$
\end{description}
Moreover, in both case, the energy function $W$ satisfies
\begin{equation}
W(t)=\circ(1/t^{1+\theta})\text{ and }\int_{t_{0}}^{+\infty}t^{\theta
}W(t)dt<+\infty. \label{rateofconvergence}%
\end{equation}
\end{theorem}

Our last mean result deals with the strong convergence of the trajectories of
(GAVD$_{\gamma,\varepsilon}$) to a minimizer of the function $U$ on the set of
minimizers of $\Phi$.

\begin{theorem}
[Strong convergence properties of (GAVD$_{\gamma,\varepsilon})$]%
\label{Strong}Assume that $U$ is strongly convex and $\gamma
(t)=\displaystyle\frac{\alpha}{t^{\theta}}$ with $\alpha>0$ if $0\leq\theta<1$ and
$\alpha>3$ if $\theta=1.$ Suppose in addition that $\displaystyle\int_{t_{0}}^{+\infty
}t^{\theta}\varepsilon(t)dt=+\infty.$ Let $x(.)$ be a classical solution of
(GAVD$_{\gamma,\varepsilon}$). Then the two following assertions hold:

\begin{description}
\item[Q$_{1}$] If $\displaystyle x^{\prime}(t)=\circ(1/t^{\theta})$ and $\displaystyle\int_{t_{0}%
}^{+\infty}t^{\theta}\left\Vert x^{\prime}(t)\right\Vert ^{2}dt<+\infty$ then
$x(t)$ converges strongly to the unique minimizer $p^{\ast}$ of $U$ on
$S_\Phi.$

\item[Q$_{2}$] If the unique minimizer $x^{\ast}$ of $U$ on $\mathcal{H}$ belongs to
$S_\Phi$ then $x(t)$ converges strongly to $x^{\ast}$ and the energy
function $W$ satisfies the asymptotic properties (\ref{rateofconvergence}).
\end{description}
\end{theorem}

\begin{remark}
In the case $\gamma(t)=\gamma$ is constant (which correspond the case
$\theta=0$), combining Theorem \ref{General} and the assertion (Q$_{1}$) of
the Theorem \ref{Strong} yields a generalization of the strong convergence
result of Attouch and Cazernicki [Theorem 2.3, \cite{AC}].
\end{remark}

\section{A general minimization property of (GAVD$_{\gamma,\varepsilon}$)}

This section is devoted to the proof of Theorem \ref{General} which is
inspired from the arguments of [Theorem 3.1, \cite{HJ}] and [Theorem 2.1, \cite{JM15}]. Notice that the assumption (\text{H}$_{2}$) can be excluded.

\begin{proof}
Differentiating the energy function $W$ and using the equation (GAVD$_{\gamma
,\varepsilon}$), we obtain%
\begin{align}
W^{\prime}(t)  &  =-\gamma(t)\left\Vert x^{\prime}(t)\right\Vert
^{2}+\varepsilon^{\prime}(t)(U(x(t))-U^{\ast})\label{B1}\\
&  \leq-\gamma(t)\left\Vert x^{\prime}(t)\right\Vert ^{2}.\nonumber
\end{align}
Hence $W(t)$ is decreasing and approaching to some nonnegative real number $W_{\infty}$ as
$t\rightarrow+\infty$. Moreover, we have%
\begin{equation}
\int_{t_{0}}^{+\infty}\gamma(t)\left\Vert x^{\prime}(t)\right\Vert
^{2}dt<\infty. \label{in}%
\end{equation}
To conclude, we just have to show that $W_{\infty}\leq0.$ Let $v$ be an
arbitrarily element of $\mathcal{H}$. Consider the function%
\[
h_{v}(t)\equiv\frac{1}{2}\left\Vert x(t)-v\right\Vert ^{2}.
\]
Using the equation (GAVD$_{\gamma,\varepsilon}$) and the convexity of $\Phi$
and $U,$ one can easily check that%
\begin{align}
h_{v}^{\prime\prime}(t)+\gamma(t)h_{v}^{\prime}(t)  &  =\left\Vert x^{\prime
}(t)\right\Vert ^{2}+\langle\nabla\Phi(x(t)),v-x(t)\rangle+\varepsilon
(t)\langle\nabla U(x(t)),v-x(t)\rangle\nonumber\\
&  \leq\left\Vert x^{\prime}(t)\right\Vert ^{2}+\Phi(v)-\Phi(x(t))+\varepsilon
(t)(U(v)-U(x(t)))\nonumber\\
&  =\frac{3}{2}\left\Vert x^{\prime}(t)\right\Vert ^{2}-W(t)+\Phi
(v)-\Phi^{\ast}+\varepsilon(t)(U(v)-U^{\ast}). \label{A}%
\end{align}
Recalling that $W(t)\geq W_{\infty},$ we get%
\[
A_{\infty}\leq-h_{v}^{\prime\prime}(t)-\gamma(t)h_{v}^{\prime}(t)+\frac{3}%
{2}\left\Vert x^{\prime}(t)\right\Vert ^{2}+\varepsilon(t)(U(v)-U^{\ast})
\]
where $\displaystyle A_{\infty}=W_{\infty}+\Phi^{\ast}-\Phi(v).$ \\
Integrating the last
inequality over $[t_{0},t]$  and using the fact that $\displaystyle\gamma
h_{v}\geq0$ and the assumption $\displaystyle\gamma(t)\leq\frac{K_{2}}{t^{2}},$ we find%
\begin{equation}
(t-t_{0})A_{\infty}\leq h_{v}^{\prime}(t_{0})-\gamma(t_{0})h_{v}(t_{0}%
)+h_{v}^{\prime}(t)+\frac{3}{2}\int_{t_{0}}^{t}\left\Vert x^{\prime
}(s)\right\Vert ^{2}ds+\int_{t_{0}}^{t}f_{v}(s)ds, \label{in1}%
\end{equation}
where $f_{v}(s)=\varepsilon(s)(U(v)-U^{\ast})+\frac{K_{2}}{s^{2}}h_{v}(s).$

From (\ref{in}), we deduce that $\displaystyle\int_{t_{0}}^{+\infty}\frac{\left\Vert
x^{\prime}(s)\right\Vert ^{2}}{s}ds<\infty$ which implies (see [Lemma
3.2, \cite{HJ}) that
\begin{equation}
\int_{t_{0}}^{t}\left\Vert x^{\prime}(s)\right\Vert ^{2}ds=\circ(t). \label{A1}%
\end{equation}
Using now the Cauchy-Schwartz inequality, we infer
\begin{align}\displaystyle
\left\Vert x(t)\right\Vert  &  \leq\left\Vert x(t_{0})\right\Vert
+\sqrt{t-t_{0}}\left(\int_{t_{0}}^{t}\left\Vert x^{\prime}(s)\right\Vert
^{2}ds\right)^{\frac{1}{2}}\nonumber\\
&=\circ(t)\label{o1}.
\end{align}
Therefore $\displaystyle\lim_{t\rightarrow+\infty}f_{v}(t)=0$ and as a consequence
\begin{equation}
\int_{t_{0}}^{t}f_{v}(s)ds=\circ(t). \label{A2}%
\end{equation}
Recalling now that since $W~$is bounded, $x^{\prime}$ is also bounded. Thus,
we get by (\ref{o1})
\begin{equation}
h_{v}^{\prime}(t)=2\langle x^{\prime}(t),x(t)-v\rangle=\circ(t). \label{A3}%
\end{equation}
Finally, dividing the inequality (\ref{in1}) by $t$, using the estimates
(\ref{A1}), (\ref{A2}), (\ref{A3}) and letting $t\rightarrow+\infty$, we
obtain $A_{\infty}\leq0,$ which implies that $W_{\infty}\leq\Phi
(v)-\Phi^{\ast}.$ Since this holds for every $v\in\mathcal{H},$ the required
result $W_{\infty}\leq0$ follows.
\end{proof}
\begin{remark}
  Let us notice that, if $S_\Phi$ is empty, any solution $x(.)$ of the (GAVD$_{\gamma,\varepsilon}$) system is unbounded. Indeed, if else, there exists a sequence $(t_n)_n$ tending to $+\infty$ so that $(x(t_n))_n$ converges weakly to an element $\bar{x}\in\mathcal{H}$. From the lower semi-continuity property it follows that $$\Phi(\bar{x})\leq\underset{n\to+\infty}{\liminf} \Phi(x(t_n)),$$
  which means that $\Phi(\bar{x})\leq \Phi^*$. This contradicts that  $S_\Phi=\emptyset.$
\end{remark}
\section{Weak convergence properties of (GAVD$_{\gamma,\varepsilon}$)}

In this section, we give a proof of Theorem \ref{Weak} which relies on the
classical Opial's lemma and the following important lemma which will be also
useful in the study of the strong convergence properties of the trajectories
of (GAVD$_{\gamma,\varepsilon}$) in the next section.

\begin{lemma}
\label{Lemma}Assume that the function $\gamma(.)$ satisfies the
assumption (\ref{assump}) in Theorem \ref{Weak}. Let $x(.)$ be a classical
solution of (GAVD$_{\gamma,\varepsilon}$) and let $v\in S_\Phi$ such that $[t^{\theta}r_{v}(t)]^{+}$  belongs to $L^{1}(t_{0},+\infty;\mathbb{R})$ where $r_{v}(t)\equiv \varepsilon(t)(U(v)-U(x(t))$.
Then the distance function $h_{v}(t)\equiv\frac{1}{2}\left\Vert x(t)-v\right\Vert ^{2}$
converges as $t\rightarrow+\infty$ and the energy function $W$ satisfies the
asymptotic property (\ref{rateofconvergence}).
\end{lemma}

\begin{proof}
First, we notice
that up to take $t_{1}$ large enough we can assume that%
\[
\gamma(t)\geq\frac{K}{t}\text{ for every }t\geq t_{1}%
\]
with $K>3$ and $K=\alpha$ if $\theta=1.$

Let $\lambda(t)=t^{1+\theta}$. Using (\ref{B1}) and the above inequality, we
find%
\begin{align}
(\lambda W)^{\prime} &  \leq\lambda^{\prime}W-\lambda\gamma\left\Vert
x^{\prime}\right\Vert ^{2}\nonumber\\
&  \leq\lambda^{\prime}W-\frac{K}{1+\theta}\lambda^{\prime}\left\Vert
x^{\prime}\right\Vert ^{2}\nonumber\\
&  \leq\lambda^{\prime}W-\frac{K}{2}\lambda^{\prime}\left\Vert x^{\prime
}\right\Vert ^{2}.\label{E}%
\end{align}
Therefore,%
\[
\frac{3}{2}\lambda^{\prime}\left\Vert x^{\prime}\right\Vert ^{2}\leq\frac
{3}{K}\lambda^{\prime}W-\frac{3}{K}(\lambda W)^{\prime}.
\]
Multiplying (\ref{A}) by $\lambda^{\prime}(t)$ (we recall that, since
$v\in S_\Phi,$ $\Phi(v)=\Phi^{\ast}$) and using the above inequality, we
obtain%
\[\displaystyle
(1-\frac{3}{K})\lambda^{\prime}W+\frac{3}{K}(\lambda W)^{\prime}\leq
-\lambda^{\prime}h_{v}^{\prime\prime}-\lambda^{\prime}\gamma h_{v}^{\prime
}+\lambda^{\prime}[r_{v}]^{+}.
\]
Integrating this inequality from $t_{1}$ to $t$ we get
\begin{equation}\displaystyle{\small
(1-\frac{3}{K})\int_{t_{1}}^{t}\lambda^{\prime}(s)W(s)ds+\frac{3}{K}%
\lambda(t)W(t)\leq C_{0}-\lambda^{\prime}(t)h_{v}^{\prime}(t)+(\lambda
^{\prime\prime}-\lambda^{\prime}\gamma)(t)h_{v}(t)+\int_{t_{1}}^{t}%
g_{\theta}(s)h_{v}(s)ds,\label{B2}}
\end{equation}
where
\[\displaystyle
C_{0}=\lambda^{\prime}(t_{1})h_{v}^{\prime}(t_{1})-\lambda^{\prime\prime
}(t_{1})h_{v}(t_{1})+\frac{3}{K}\lambda(t_1)W(t_1)+\lambda^{\prime}(t_1)h_{v}^{\prime}(t_1)+\int_{t_{1}}^{+\infty}\lambda^{\prime}(s)[r_{v}(s)]^{+}ds
\]
and
\begin{equation}
g_{\theta}(t)=[(\lambda^{\prime}\gamma)^{\prime}]^{+}(t)-\lambda
^{\prime\prime\prime}(t).\label{G}%
\end{equation}
Let $A(\theta)$ and $\mu(\theta)>0$ be two positive constants such that $A(\theta)+\mu(\theta
)<(\theta+1)\alpha$ if $\theta<1$ and $A(\theta)+\mu(\theta)=2(\alpha-1)$ if
$\theta=1.$ Since
\[
(\lambda^{\prime\prime}-\lambda^{\prime}\gamma)(t)\leq(1+\theta)(\theta
t^{\theta-1}-\alpha),
\]
we can assume, up to take $t_{1}$ large enough in the case $\theta<1$, that
\begin{equation}
(\lambda^{\prime\prime}-\lambda^{\prime}\gamma)(t)\leq-A(\theta)-\mu
(\theta)\text{ }\forall t\geq t_{1}.\label{H}%
\end{equation}
Using now the fact that
\begin{align*}
\left\vert h_{v}^{\prime}(t)\right\vert  &  \leq\left\Vert x^{\prime
}(t)\right\Vert \left\Vert x(t)-v\right\Vert \\
&  \leq2\sqrt{W(t)}\sqrt{h_{v}(t)},
\end{align*}
it follows, from the estimate (\ref{H}) and the elementary inequality%
\[
bx-ax^{2}\leq\frac{b^{2}}{4a}\ \ \ \forall a>0,\ (x,b)\in\mathbb{R}^{2},
\]
that for every $t\geq t_{1}$%
\begin{align}
-\lambda^{\prime}(t)h_{v}^{\prime}(t)+(\lambda^{\prime\prime}-\lambda^{\prime
}\gamma)(t)h_{v}(t) &  \leq\frac{(\lambda^{\prime}(t))^{2}W(t)}{A(\theta)}%
-\mu(\theta)h_{v}(t)\nonumber\\
&  =B(\theta,t)\lambda(t)W(t)-\mu(\theta)h_{v}(t)\label{H2},%
\end{align}
where%
\[
B(\theta,t)=\frac{(\theta+1)^{2}t^{\theta-1}}{A(\theta)}.
\]
Inserting (\ref{H2}) in the inequality (\ref{B2}), we obtain%
\begin{equation}
(1-\frac{3}{K})\int_{t_{1}}^{t}\lambda^{\prime}(s)W(s)ds+(\frac{3}{K}%
-B(\theta,t))\lambda(t)W(t)+\mu(\theta)h_{v}(t)\leq C_{0}+\int_{t_{1}}%
^{t}g_{\theta}(s)h_{v}(s)ds\label{H3}%
\end{equation}
Let us notice that if $0\leq \theta<1$ then $\displaystyle\lim_{t\to+\infty}B(\theta,t)=0$ and in the case where $\theta=1$, since $\alpha>3$, one can choose $\displaystyle 0<\mu(1)<\frac{2}{3}(\alpha-3)$ to get
\begin{eqnarray*}
\frac{3}{K}-B(1,t) &=&\frac{3}{\alpha}- \frac{4}{A(1)}>0.
\end{eqnarray*}
Hence, up to take $t_{1}$ large enough we assume that, for every $0 \leq\theta\leq 1$, there exists a constant $\nu(\theta)>0$ such that%
\[
\frac{3}{K}-B(\theta,t)\geq\nu(\theta), \mbox{ for all } t\geq t_1.
\]
Recalling that the function $g_{\theta}$ is integrable over $[t_{1}%
,+\infty)$ and applying the Gronwall lemma to the inequality (\ref{H3}),
we deduce that the function $h_{v}$ is bounded and as a consequence we get%
\begin{align}
\sup_{t\geq t_{1}}\lambda(t)W(t) &  <+\infty\nonumber
\end{align}
and
\begin{align}
\int_{t_{1}}^{+\infty}\lambda^{\prime}(s)W(s)ds &  <+\infty.\label{Er}%
\end{align}
Now, using the fact that the energy function $W$ is decreasing, we deduce from
(\ref{Er}) that $t^{1+\theta}W(t)\rightarrow0$ as $t\rightarrow+\infty$ in
fact for every $t\geq t_{1}$ we have $$\displaystyle (1+\theta)\left(  \frac{t}{2}\right)
^{1+\theta}W(t)\leq\int_{\frac{t}{2}}^{t}\lambda^{\prime}(s)W(s)ds.$$ To
conclude, it remains to prove that $\displaystyle\lim_{t\rightarrow+\infty}h_{v}(t)$ exists.
From (\ref{A}), the function $h_{v}$ satisfies the differential inequality%
\[\displaystyle
h_{v}^{\prime\prime}(t)+\gamma(t)h_{v}^{\prime}(t)\leq\zeta(t)
\]
where $\zeta(t)=\frac{3}{2}\left\Vert x^{\prime}(t)\right\Vert ^{2}%
+[r_{v}(t)]^{+}.$ The assumption on the function $r_{v}$ and the estimate
(\ref{Er}) imply that $t^{\theta}\zeta(t)\in L^{1}(a,+\infty;\mathbb{R}^{+}),$
then the existence of $\displaystyle\lim_{t\rightarrow+\infty}h_{v}(t)$ follows from the
following lemma.
\end{proof}

\begin{lemma}
\label{lemma}Let $a>0$ and $w:[a,+\infty)\rightarrow\mathbb{R}^{+}$ be a
continuous function satisfying%
\[
w(t)\geq\frac{\alpha}{t^{\theta}}~\forall t\geq a
\]
where $\alpha$ and $\theta$ are nonnegative constants with $0\leq\theta\leq1$
and $\alpha>1$ if $\theta=1.$ Let $\varphi\in C^{2}(a,+\infty;\mathbb{R}^{+})$
satisfy the differential inequality%
\begin{equation}
\varphi^{\prime\prime}(t)+w(t)\varphi^{\prime}(t)\leq\psi(t) \label{Ine}%
\end{equation}
with $t^{\theta}\psi(t)\in L^{1}(a,+\infty;\mathbb{R}^{+}).$ Then
$\displaystyle\lim_{t\rightarrow+\infty}\varphi(t)$ exists.
\end{lemma}

\begin{proof}
From (\ref{Ine}), we have for every $t\geq a$
\begin{equation}
\varphi^{\prime}(t)\leq e^{-\Gamma(t,a)}\varphi^{\prime}(a)+\int_{a}%
^{t}e^{-\Gamma(t,s)}\psi(s)ds; \label{In}%
\end{equation}
where%
\[
\Gamma(t,s)=\int_{s}^{t} w(\tau)d\tau.
\]
Similarly to as in the proof of [Lemma 3.14,\cite{CF}], one can easily check
that
\[\displaystyle
\int_{s}^{+\infty}e^{-\Gamma(t,s)}dt\leq M~s^{\theta}~\forall s\geq a,
\]
where $M>0$ is an absolute constant. We deduce from (\ref{In}) and
Fubini's Theorem that the positive part $[\varphi^{\prime}]^{+}$ of
$\varphi^{\prime}$ belongs to $L^{1}(a,+\infty;\mathbb{R}^{+})$ which implies
that $\displaystyle\lim_{t\rightarrow+\infty}\varphi(t)$ exists.
\end{proof}

Before starting the proof of Theorem \ref{Weak}, let us recall the classical
Opial's lemma.

\begin{lemma}
[Opial's lemma]Let $x:[t_{0},+\infty)\rightarrow\mathcal{H}.$ Assume
that there exists a nonempty subset $S$ of $\mathcal{H}$ such that:

\begin{enumerate}
\item[i)] if $t_{n}\rightarrow+\infty$ and $x(t_{n})\rightharpoonup x$ weakly
in $\mathcal{H}$ , then $x\in S$,

\item[ii)] for every $z\in S,$ $\displaystyle\lim_{t\rightarrow+\infty}\left\Vert
x(t)-z\right\Vert $ exists.
\end{enumerate}

\noindent Then there exists $z_{\infty}\in S$ such that $x(t)\rightharpoonup
z_{\infty}$ weakly in $\mathcal{H}$ as $t\rightarrow+\infty.$
\end{lemma}

For a simple proof of Opial's lemma, we refer the reader to \cite{Op}.

\begin{proof}
[Proof of Theorem \ref{Weak}]\textbf{Step 1}: Proof of the property (P$_{1}$).
Since $r_{v}(t)\leq\varepsilon(t)(U(v)-U^{\ast}),$ then, according to Lemma
\ref{Lemma}, $\displaystyle\lim_{t\rightarrow+\infty}h_{v}(t)$ exists for every $v\in
S_\Phi$ and the energy function $W$ satisfies (\ref{rateofconvergence}%
)$.$ Let $t_{n}\rightarrow+\infty$ such that $x(t_{n})$ converges weakly in
$\mathcal{H}$\ to some $\bar{x}.$ Since $\Phi(x(t))\rightarrow\Phi^{\ast}$ as
$t\rightarrow+\infty,$ the weak lower semi-continuity of $\Phi$ implies that
$\Phi(\bar{x})\leq\Phi^{\ast}$ which means that $\bar{x}\in S_\Phi.$  By Opial's lemma,
we deduce that $x(t)$ converges weakly in
$\mathcal{H}$\ as $t\rightarrow+\infty$ to some element of $S_\Phi.$

\textbf{Step2:} Proof of the property (P$_{2}$). Let $v\in S=S_\Phi\cap S_U.$ Since $r_{v}$ is  nonpositive, then Lemma \ref{Lemma} implies that
$\displaystyle\lim_{t\rightarrow+\infty}h_{v}(t)$ exists and $W$ satisfies
(\ref{rateofconvergence})$.$ Thus, in view of the assumption $\displaystyle\liminf_{t\rightarrow+\infty
}t^{\theta+1}\varepsilon(t)>0$, we have $U(x(t))\rightarrow U^{\ast}$ as
$t\rightarrow+\infty.$ Therefore the lower semi-continuity of $\Phi$ and $U$
gives, as in the above step, that every sequential weak cluster point of
$x(t),~$as $t\rightarrow+\infty,$ belongs to the subset $S.$ This completes
the proof of the property (P$_{2}$) due to Opial's lemma.
\end{proof}

\section{Strong convergence properties of (GAVD$_{\gamma,\varepsilon}$)}

This section is devoted to the proof of Theorem \ref{Strong}. Before Proving
separably the two properties (Q$_{1}$) and (Q$_{2}$), let us first recall some
general facts about strongly convex functions and the Tikhonov approximation
method \cite{TA}. The function $U$ is strongly convex then there exists a positive real
$m$ such that $U(x)-\frac{m}{2}\left\Vert x\right\Vert ^{2}$ is convex
(we say that $U$ is $m-$strongly convex). Moreover, for every nonempty, convex
and closed subset $C$ of $\mathcal{H}$, the function $U$ has a unique
minimizer $x_{C}^{\ast}$ on $C.$ Let $x^{\ast}$ be the minimizer on
$\mathcal{H}$ and $p^{\ast}$ its minimizer on $S_\Phi.$ For every $t\geq
t_{0},$ we consider the function $\Phi_{t}$ defined on $\mathcal{H}$ by
\[
\Phi_{t}(x)=\Phi(x)+\varepsilon(t)U(x).
\]
Clearly, $\Phi_{t}$ is $\varepsilon(t)m$-strongly convex. Therefore, $\Phi
_{t}$ satisfies the convex inequality
\begin{equation}
\Phi_{t}(z)\geq\Phi_{t}(y)+\langle\nabla\Phi_{t}(y),z-y\rangle+\frac{m}%
{2}\varepsilon(t)\left\Vert z-y\right\Vert ^{2}, \label{Cineq}%
\end{equation}
and has a unique global minimizer which we denote by $x_{\varepsilon(t)}.$
Adopting the Tikhonov method, we can prove that $x_{\varepsilon(t)}$
converges strongly to $p^{\ast}$ as $t\rightarrow+\infty.$ Indeed, since%
\begin{equation}
\Phi_{t}(x_{\varepsilon(t)})\leq\Phi_{t}(p^{\ast}) \label{T1}%
\end{equation}
and%
\[
\Phi(p^{\ast})\leq\Phi(x_{\varepsilon(t)}),
\]
then
\begin{equation}
U(x_{\varepsilon(t)})\leq U(p^{\ast}). \label{T2}%
\end{equation}
Furthermore, seeing that $U$ is coercive, the last inequality implies that $(x_{\varepsilon
(t)})_{t\geq t_{0}}$ is bounded. So, let
$\tilde{x}\in\mathcal{H}$ \ be a weak limit of a sequence $(x_{\varepsilon
(t_{n})})$ where $t_{n}\rightarrow+\infty$. Using the weak lower
semi-continuity of the two convex functions $\Phi$ and $U$ and letting
$t=t_{n}\rightarrow+\infty$ in the inequalities (\ref{T1}) and (\ref{T2}), we
deduce that $\Phi(\tilde{x})\leq\Phi(p^{\ast})$ and $U(\tilde{x})\leq
U(p^{\ast})$ which is, from the definition of $p^{\ast},$ is equivalent to
$\tilde{x}=p^{\ast}.$ Consequently, we infer that $x_{\varepsilon(t)}$
converges weakly to $p^{\ast}$ as $t_{n}\rightarrow+\infty$. Now, for the reason that $U$ is
$m-$strongly convex, we have
\[
U(x_{\varepsilon(t)})\geq U(p^{\ast})+\langle\nabla U(p^{\ast}),x_{\varepsilon
(t)}-p^{\ast}\rangle+\frac{m}{2}\left\Vert x_{\varepsilon(t)}-p^{\ast
}\right\Vert ^{2}.
\]
Hence, by (\ref{T2}), we deduce that
$\displaystyle\lim_{t\rightarrow+\infty}\left\Vert x_{\varepsilon(t)}-p^{\ast}\right\Vert
=0$ which completes the proof of the claim.

\begin{proof}
[Proof of Theorem \ref{Strong}]Let us first prove the property (Q$_{1}$). We
consider the function $h(t)=h_{p^{\ast}}(t)=\frac{1}{2}\left\Vert
x(t)-p^{\ast}\right\Vert ^{2}.$ Using the equation (GAVD$\gamma,\varepsilon$)
and the convex inequality (\ref{Cineq}) we obtain%
\begin{align}
h^{\prime\prime}(t)+\gamma(t)h^{\prime}(t) &  =\left\Vert x^{\prime
}(t)\right\Vert ^{2}+\langle\nabla\Phi_{t}(x(t)),p^{\ast}-x(t)\rangle
\nonumber\\
&  \leq\left\Vert x^{\prime}(t)\right\Vert ^{2}+\Phi_{t}(p^{\ast})-\Phi
_{t}(x(t))-m~\varepsilon(t)h(t)\nonumber\\
&  \leq\left\Vert x^{\prime}(t)\right\Vert ^{2}+\Phi_{t}(p^{\ast})-\Phi
_{t}(x_{\varepsilon(t)})-m~\varepsilon(t)h(t)\nonumber\\
&  \leq\left\Vert x^{\prime}(t)\right\Vert ^{2}+\varepsilon(t)(U(p^{\ast
})-U(x_{\varepsilon(t)}))-m~\varepsilon(t)h(t).\label{Lineq}%
\end{align}
In the last inequality we have used the fact that $p^{\ast}$ is also a
minimizer of $\Phi.$ Set
\[
\sigma(t)\equiv U(x_{\varepsilon(t)})-U(p^{\ast})+m~h(t).
\]
The inequality (\ref{Lineq}) becomes%
\begin{equation}
h^{\prime\prime}(t)+\gamma(t)h^{\prime}(t)+\varepsilon(t)\sigma(t)\leq
\left\Vert x^{\prime}(t)\right\Vert ^{2}.\label{Pineq}%
\end{equation}
Let us prove that $\displaystyle\liminf_{t\rightarrow+\infty}h(t)=0.$ We argue by
contradiction. As consequence of $$\displaystyle\lim_{t\rightarrow+\infty}U(x_{\varepsilon
(t)})-U(p^{\ast})=0,$$ there exists $t_{2}\geq t_{0}$ large enough and $\mu>0$
such that $\sigma(t)\geq\mu$ for every $t\geq t_{2}.$ Therefore the
differential inequality (\ref{Pineq}) implies that, for every $t\geq t_{2}$, we
have%
\[
h(t)+\mu\int_{t_{2}}^{t}\int_{t_{2}}^{\tau}e^{-\Gamma(\tau,s)}\varepsilon
(s)dsd\tau\leq h(t_{2})+\int_{t_{2}}^{t}e^{-\Gamma(\tau,t_{2})}d\tau
h^{\prime}(t_{2})+\int_{t_{2}}^{t}\int_{t_{2}}^{\tau}e^{-\Gamma(\tau
,s)}\left\Vert x^{\prime}(s)\right\Vert ^{2}dsd\tau,
\]
where%
\[
\Gamma(t,s)=\int_{s}^{t}\gamma(\tau)d\tau.
\]
Applying Fubini's theorem, we then infer that%
\begin{equation}
\mu\int_{t_{2}}^{+\infty}\varepsilon(s)\int_{s}^{+\infty}e^{-\Gamma(\tau
,s)}d\tau ds\leq h(t_{2})+\left\vert h^{\prime}(t_{2})\right\vert \int_{t_{2}%
}^{+\infty}e^{-\Gamma(\tau,t_{2})}d\tau+\int_{t_{2}}^{+\infty}\left\Vert
x^{\prime}(s)\right\Vert ^{2}\int_{s}^{+\infty}e^{-\Gamma(\tau,s)}d\tau
ds.\label{Rineq}%
\end{equation}
A simple computation ensures the existence of two real constants $B_{\theta
}>A_{\theta}>0$ so that%
\[
A_{\theta}~s^{\theta}\leq\int_{s}^{+\infty}e^{-\Gamma(\tau,s)}d\tau\leq
B_{\theta}~s^{\theta}.
\]
Hence, combining the inequality (\ref{Rineq}) and the assumption $\displaystyle\int_{t_{0}}^{+\infty
}s^{\theta}\left\Vert x^{\prime}(s)\right\Vert ^{2}ds<+\infty$, we get
$\displaystyle\int_{t_{0}}^{+\infty}s^{\theta}\varepsilon(s)ds<+\infty,$ a contradiction.
Consequently
\begin{equation}
\liminf_{t\rightarrow+\infty}h(t)=0.\label{J1}%
\end{equation}
Now let us suppose that
\begin{equation}
\limsup_{t\rightarrow+\infty}h(t)>0.\label{J2}%
\end{equation}
The continuity of the function
$h$ combined with (\ref{J1}) and (\ref{J2}) ensure the existence of two real
numbers $\lambda<\delta$ and two positive real sequences $(s_{n})_{n}$ and
$(t_{n})_{n}$ such that for every $n\in\mathbb{N}$ we have
\begin{align*}
\max\{t_{\ast},n\} &  <s_{n}<t_{n},\\
h(t_{n}) &  =\delta,\\
h(s_{n}) &  =\lambda,\\
h(s) &  \in\lbrack\lambda,\delta]\text{ on }[s_{n},t_{n}],
\end{align*}
where $t_{\ast}>t_{2}$ is a fixed positive number such that $U(x_{\varepsilon
(t)})-U(p^{\ast})\geq-m\lambda$  for all $t\geq t_{\ast}$ (for more details see [Theorem 5.1 \cite{JM10}]). We deduce from (\ref{Pineq}) that for
every $n\in\mathbb{N}$ and for all $t\in\lbrack s_{n},t_{n}]$ %
\[
h^{\prime\prime}(t)+\frac{\alpha}{t^{\theta}}h^{\prime}(t)\leq\left\Vert
x^{\prime}(t)\right\Vert ^{2}.%
\]
Multiplying the last differential
inequality by $t^{\theta}$ and integrating over $[s_{n},t_{n}],$ we obtain
\begin{equation}
t_{n}^{\theta}h^{\prime}(t_{n})-s_{n}^{\theta}h^{\prime}(s_{n})+\theta
s_{n}^{\theta-1}\lambda-\theta t_{n}^{\theta-1}\delta+\alpha(\delta
-\lambda)+\theta(\theta-1)\int_{s_{n}}^{t_{n}}t^{\theta-2}h(t)dt\leq
\int_{s_{n}}^{t_{n}}t^{\theta}\left\Vert x^{\prime}(t)\right\Vert
^{2}.\label{Chineq}%
\end{equation}
Using now the facts%
\begin{align*}
\left\vert h^{\prime}(t_{n})\right\vert  &  \leq\left\Vert x^{\prime}%
(t_{n})\right\Vert \sqrt{2h(t_{n})}=\left\Vert x^{\prime}(t_{n})\right\Vert
\sqrt{2\delta},\\
\left\vert h^{\prime}(s_{n})\right\vert  &  \leq\left\Vert x^{\prime}%
(s_{n})\right\Vert \sqrt{2\lambda},\\
\int_{s_{n}}^{t_{n}}t^{\theta-2}h(t)dt &  \leq\delta\frac{s_{n}^{\theta-1}%
}{1-\theta}\text{ if }0\leq\theta<1,
\end{align*}
and letting $n$ goes to $+\infty$ in the the inequality (\ref{Chineq}), we
get
\begin{align*}
(\alpha-1)(\delta-\lambda) &  \leq0\text{ if }\theta=1,\\
\alpha(\delta-\lambda) &  \leq0\text{ if }0\leq\theta<1.
\end{align*}
This contradicts the assumption $\delta>\lambda.$ We therefore conclude that
$\displaystyle\lim_{t\rightarrow+\infty}h(t)=0,$ which completes the proof of the property
(Q$_{1}$).
\end{proof}
\section{Numerical Experiments}
\textbf{Acknowledgement:} The authors are grateful to the Deanship of Scientific Research at King Faisal University for financially and morally supporting this work under Project 170065​.

\end{document}